\documentclass[a4paper,12pt]{amsart}
\usepackage{amsmath}
 \usepackage{latexsym, amsthm, amscd, euscript, float, subfig}
\usepackage{amssymb}
\usepackage{ifthen}
\usepackage{graphicx}
\usepackage{float}
\usepackage{cite}
\usepackage{amsfonts}
\usepackage{amscd}
\usepackage{amsxtra}
\usepackage{color}

\newtheorem{theorem}{Theorem}[section]
\newtheorem{lemma}[theorem]{Lemma}
\newtheorem{corollary}[theorem]{Corollary}

\theoremstyle{definition}

\numberwithin{equation}{section}

\def\be{\begin{equation}}
\def\ee{\end{equation}}

\newcounter{alphabet}


\begin{document}
\bibliographystyle{amsplain}
\title[Second Hankel determinant of logarithmic inverse coefficients]{Second Hankel Determinant for Logarithmic Inverse Coefficients of Strongly Convex and Strongly Starlike Functions}
\author[Vasudevarao Allu]{Vasudevarao Allu}
\address{Vasudevarao Allu, School of Basic Sciences, Indian Institute of Technology Bhubaneswar,
Bhubaneswar-752050, Odisha, India.}
\email{avrao@iitbbs.ac.in}
\author[Amal Shaji]{Amal Shaji}
\address{Amal Shaji, School of Basic Sciences, Indian Institute of Technology Bhubaneswar,
Bhubaneswar-752050, Odisha, India.}
\email{amalmulloor@gmail.com}

\begin{abstract}
In this paper, we obtain the sharp bounds of the second Hankel determinant of logarithmic  inverse coefficients  for the strongly starlike and strongly convex functions of order alpha.
\end{abstract}

\maketitle

\section{Introduction}\label{Introduction}
Let $\mathcal{H}$ denote the class of analytic functions in the unit disk $\mathbb{D}:=\{z\in\mathbb{C}:\, |z|<1\}$. Here $\mathcal{H}$ is 
a locally convex topological vector space endowed with the topology of uniform convergence over compact subsets of $\mathbb{D}$. Let $\mathcal{A}$ denote the class of functions $f\in \mathcal{H}$ such that $f(0)=0$ and $f'(0)=1$.  Let $\mathcal{S}$ 
denote the subclass of  $\mathcal{A}$ consisting of functions which are univalent ({\em i.e., one-to-one}) in $\mathbb{D}$. 
If $f\in\mathcal{S}$ then it has the following series representation
\begin{equation}\label{f}
	f(z)= z+\sum_{n=2}^{\infty}a_n z^n, \quad z\in \mathbb{D}.
\end{equation}

The {\it Logarithmic coefficients} $\gamma_{n}$ of $f\in \mathcal{S}$ are defined by,
\begin{equation}\label{amal-1}
	F_{f}(z):= \log\frac{f(z)}{z}=2\sum\limits_{n=1}^{\infty}\gamma_{n}z^{n}, \quad z \in \mathbb{D}.
\end{equation}
The logarithmic coefficients $\gamma_{n}$ play a central role in the theory of univalent functions. A very few exact upper bounds for $\gamma_{n}$ seem to have been established. The significance of this problem in the context of Bieberbach conjecture was pointed by Milin\cite{milin} in his conjecture. Milin \cite{milin} has conjectured that for $f\in \mathcal{S}$ and $n\ge 2$, 
$$\sum\limits_{m=1}^{n}\sum\limits_{k=1}^{m}\left(k|\gamma_{k}|^{2}-\frac{1}{k}\right)\le 0,$$
which led De Branges, by proving this conjecture, to the proof of Bieberbach conjecture  \cite{De Branges-1985}. For the Koebe function $k(z)=z/(1-z)^{2}$, the logarithmic coefficients are $\gamma_{n}=1/n$. Since the Koebe function $k$ plays the role of extremal function for most of the extremal problems in the class $\mathcal{S}$, it is expected that $|\gamma_{n}|\le1/n$ holds for functions in $\mathcal{S}$. But this is not true in general, even in order of magnitude. By differentiating \eqref{amal-1} and  the equating coefficients we obtain
\begin{equation}\label{gamma}
\begin{aligned}
& \gamma_{1}=\frac{1}{2}a_{2}, \\[2mm]
& \gamma_{2}=\frac{1}{2}(a_{3}-\frac{1}{2}a_{2}^{2}),\\[2mm]
& \gamma_{3}=\frac{1}{2}(a_{4}-a_{2}a_{3}+\frac{1}{3}a_{2}^{3}).
\end{aligned}
\end{equation}
If $f\in \mathcal{S}$, it is easy to see that $|\gamma_{1}|\le 1$, because $|a_2| \leq 2$. Using the Fekete-Szeg$\ddot{o}$ inequality \cite[Theorem 3.8]{Duren-book-1983} for functions in $\mathcal{S}$ in (1.4), we obtain the sharp estimate 
$$|\gamma_{2}|\le\frac{1}{2}\left(1+2e^{-2}\right)=0.635\ldots.$$
For $n\ge 3$, the problem seems much harder, and no significant bound for $|\gamma_{n}|$ when $f\in \mathcal{S}$ appear to be known. In 2017, Ali and Allu\cite{vasu2017} obtained the initial logarithmic coefficients bounds for close-to-convex functions. The problem of computing the bound of the logarithmic coefficients is also considered in \cite{cho,PSW20,vasu-2018,Thomas-2016} for several subclasses of close-to-convex functions.
\\

For $q,n \in \mathbb{N}$, the Hankel determinant $H_{q,n}(f)$ of Taylor’s coefficients of function $f \in \mathcal{A}$ of the form \eqref{f} is defined by
$$
H_{q,n}(f) =
\begin{vmatrix}
	a_n & a_{n+1}  & \cdots & a_{n+q-1} \\ 
	a_{n+1} & a_{n+2}  & \cdots & a_{n+q} \\
	\vdots & \vdots &  \ddots &\vdots \\
	a_{n+q-1} & a_{n+q}  & \cdots & a_{n+2(q-1)}
\end{vmatrix}.
$$
The Hankel determinant for various order is also studied recently by several authors in different contexts;
for instance see \cite{Pom66,Pom67,ALT}.
One can easily observe that the Fekete-Szeg\"{o} functional is the second Hankel determinant $H_{2,1}(f)$. Fekete-Szeg$\ddot{o}$  then
further generalized the estimate $|a_3 - \mu a_2
^2|$ with $\mu$ real for $f$ given by \eqref{f} (see \cite[Theorem 3.8]{Duren-book-1983}). \\[2mm]
Let $g$ be the inverse function of $f\in \mathcal{S}$  defined in a neighborhood of the
origin with the Taylor series expansion
\begin{equation}\label{inverse}
g(w)=f^{-1}(w)=w+\sum_{n=2}^{\infty}A_n w^n,
\end{equation}
where we may choose $|w| < 1/4$, as we know from Koebe’s $1/4$-theorem. Using
 variational method, L\"{o}wner \cite{Lowner} obtained the sharp estimate:
$$
|A_n| \leq K_n  \quad \text{for each}\,\,\, n \in \mathbb{N}$$
where $K_n = (2n)!/(n!(n + 1)!)$ and $K(w) = w + K_2w_2 + K_3w_3 + \cdots $ is the
inverse of the Koebe function. There has been a good deal of interest in determining the behaviour of the inverse coefficients of $f$ given in \eqref{f} when the
corresponding function $f$ is restricted to some proper geometric subclasses
of $\mathcal{S}$.\\[2mm]
Let $f(z)= z+\sum_{n=2}^{\infty}a_n z^n $ be a function in  class $\mathcal{S}$. Since $f(f^{-1})(w)=w$ and using \eqref{inverse}, it follows that
\begin{equation}\label{inversetaylor}
\begin{aligned}
& A_{2}=-a_2,\\[2mm]
& A_{3}=-a_3+2a_2^2,\\[2mm]
& A_{4}=-a_4+5a_2a_3-5a_2^3.
\end{aligned}
\end{equation}
\\

The notion of logarithmic inverse coefficients, {\em i.e.}, logartithmic coefficients of inverse of $f$, was proposed by ponnusamy {\it et al.} \cite{samyinverselog}. The {\it logarithmic
inverse coefficients} $\Gamma_n$, $n \in \mathbb{N}$, of $f$ are defined by the equation
\begin{equation}\label{Gamma}
	F_{f^{-1}}(w):= \log\frac{f^{-1}(w)}{w}=2\sum\limits_{n=1}^{\infty}\Gamma_{n}w^{n}, \quad |w|<1/4.
\end{equation}
By differentiating \eqref{Gamma} together with \eqref{inversetaylor}, we get 
\begin{equation}\label{Gammaaa}
\begin{aligned}
& \Gamma_{1}=-\cfrac{1}{2}a_2,\\[2mm]
& \Gamma_2=-\cfrac{1}{2}a_3+\cfrac{3}{4}a_2^2,\\[2mm]
& \Gamma_3=-\cfrac{1}{2}a_4+2a_2a_3-\cfrac{5}{3}a_2^3.
\end{aligned}
\end{equation}
In \cite{samyinverselog} Ponnusamy {\it et al.} found the sharp upper bound  for the logarithmic inverse coefficients for the class $\mathcal{S}$. In fact ponnusamy {\it et al.} \cite{samyinverselog} proved that when $f\in \mathcal{S}$, $$|\Gamma_n| \leq \frac{1}{2n}
  \left(
  \begin{matrix}
    2n \\
    n
  \end{matrix}
  \right),
  \quad n \in \mathbb{N}
$$
and equality holds only for the Koebe function or one of its rotations. Further, ponnusamy {\it et al.} \cite{samyinverselog} obtained sharp bound for the initial logarithmic inverse coefficients for some of the important geometric subclasses of $
\mathcal{S}$.\\

Recently, Kowalczyk and Lecko \cite{adam} together have proposed the study of the Hankel determinant whose entries are logarithmic coefficients of $ f \in \mathcal{S}$, which is given by

$$
H_{q,n}(F_f/2) =
\begin{vmatrix}
	\gamma_n & \gamma_{n+1}  & \cdots & \gamma_{n+q-1} \\ 
	\gamma_{n+1} & \gamma_{n+2}  & \cdots & \gamma_{n+q} \\
	\vdots & \vdots &  \ddots &\vdots \\
	\gamma_{n+q-1} & \gamma_{n+q}  & \cdots & \gamma_{n+2(q-1)}
\end{vmatrix}.
$$
Kowalczyk and Lecko \cite{adam} have obtained the sharp bound of the second Hankel determinant of $F_f/2$, {\em i.e.,} $H_{2,1}(F_f/2)$ for starlike and convex functions. The problem of computing the sharp bounds of $H_{2,1}(F_f/2)$ has been considered by many authors for various subclasses of $\mathcal{S}$ (See \cite{vibhuthi,vibhuthi2,adam2,Mundalia}). 
\\

In this paper, we consider the notion of the second Hankel determinant for logarithmic inverse coefficients.
Let $ f \in \mathcal{S}$ given by \eqref{f}, then the second Hankel determinant of $F_{f^{-1}}/2$ by using \eqref{Gammaaa}, is given by

\begin{equation}\label{hankel}
\begin{aligned}
	H_{2,1}(F_{f^{-1}}/2)
&=\Gamma_1\Gamma_3-\Gamma_{2}^2 \\[2mm]
&=\frac{1}{4}\left(A_2A_4-A_3^2+\frac{1}{4}A_2^4\right)\\[2mm]
&=\frac{1}{48}\left(13a_2^4-12a_2^2a3-12a_3^2+12a_2a_4\right).
\end{aligned}
\end{equation}
\\[1mm]
It is now appropriate to remark that $H_{2,1}(F_{f^{-1}}/2)$ is invariant under rotation, since for $f_{\theta}(z):=e^{-i \theta} f\left(e^{i \theta} z\right), \theta \in \mathbb{R}$ when $f \in \mathcal{S}$ we have
\begin{equation}\label{invariance}
H_{2,1}(F_{f_{\theta}^{-1}}/2)=\frac{e^{4 i \theta}}{48}\left(13a_2^4-12a_2^2a3-12a_3^2+12a_2a_4\right)=e^{4 i \theta} H_{2,1}(F_{f^{-1}}/2) .
\end{equation}
In this paper we find sharp upper bounds for $|H_{2,1}(F_{f^{-1}}/2)|
$ when $f$ belongs to the class of strongly convex or strongly starlike function of order alpha.
Given $\alpha \in (0,1]$, a function $f \in \mathcal{A}$ is called strongly convex of order $\alpha$ if
\begin{equation}\label{convexdef}
\left|\text{arg} \left(1+\cfrac{zf''(z)}{f'(z)}\right)\right|< \cfrac{\pi \alpha}{2}
\end{equation}
The set of all such functions denoted  by $\mathcal{K}_{\alpha}$. Also, a function $f \in \mathcal{A}$ is called strongly starlike of order $\alpha$ if
\begin{equation}\label{stardef}
\left|\text{arg} \left(\cfrac{zf'(z)}{f(z)}\right)\right|< \cfrac{\pi \alpha}{2}
\end{equation}

\noindent The set of all such functions denoted by $\mathcal{S}^*_{\alpha}$. The notion of strongly starlike functions was introduced by Stankiewicz \cite{Stankiewicz1} and s independently by Brannan and Kirwan\cite{Brannan}. An external geometric characterisation of strongly starlike functions was proposed by Stankiewicz \cite{Stankiewicz2}. Brannan and Kirwan\cite{Brannan} found  a geometrical condition called $\delta-$ visibility which is sufficient for functions to be starlike.

\section{Preliminary Results}
In this section, we present the key lemmas which will be used to prove the main results of this paper. Let $\mathcal{P}$ denote the class of all analytic functions $p$ having positive real part in $\mathbb{D}$, with the form
\begin{equation}\label{p}
p(z)=1+c_{1} z+c_{2} z^{2}+c_{3} z^{3}+ \cdots .
\end{equation}
A member of $\mathcal{P}$ is called a Carathéodory function. It is known that $\left|c_{n}\right| \leq 2, n \geq 1$ for a function $p \in \mathcal{P}$. By using \eqref{convexdef} and \eqref{stardef}, functions in the classes $\mathcal{S}^*_{\alpha}$ and $\mathcal{K}_{\alpha}$ can be represented interms of functions in Carathéodory class $\mathcal{P}$. To prove our main result, we need some preliminary lemmas. The first one is known as Caratheodory's lemma and the second one is due to Libera and Zlotkiewicz.
\begin{lemma}\cite{Duren-book-1983}\label{pp}
For a function $p \in \mathcal{P}$ of the form \eqref{p}, the sharp inequality holds for each $n\geq 1$. Equality holds for the function $p(z)=1+z/1-z$.
\end{lemma}
\begin{lemma}\cite{LZ1,LZ2}\label{caratheodary}
If $p \in \mathcal{P}$ is of the form \eqref{p} with $c_1 \geq 0$. Then there exits $z,w \in \mathbb{D}$ such that
\begin{equation*}
 2c_{2}=c_1^2+(4-c_1^2)z
 \end{equation*}
and 
 
 \begin{equation*}
 4c_3=c_1^3+2(4-c_1^2)c_1z-c_1(4-c_1^2)z+2(4-c_1^2)(1-|x|^2)w.
 \end{equation*}

\end{lemma}

\noindent Next we recall the following well-known result due to Choi {\it et al.} \cite{choi}. Lemma \ref{y(a,b,c)} plays an important role in the proof of our main results.

\begin{lemma}\label{y(a,b,c)}
Let $A, B, C$ be real numbers and

$$
Y(A, B, C):=\max _{z \in \overline{\mathbb{D}}}\left(\left|A+B z+C z^{2}\right|+1-|z|^{2}\right) .
$$

(i) If $A C \geq 0$, then

$$
Y(A, B, C)= \begin{cases}|A|+|B|+|C|, & |B| \geq 2(1-|C|), \\[2mm] 1+|A|+\cfrac{B^{2}}{4(1-|C|)}, & |B|<2(1-|C|) .\end{cases}
$$

(ii) If $A C<0$, then

$$
Y(A, B, C)= \begin{cases}1-|A|+\cfrac{B^{2}}{4(1-|C|)}, & -4 A C\left(C^{-2}-1\right) \leq B^{2} \wedge|B|<2(1-|C|), \\[2mm] 1+|A|+\cfrac{B^{2}}{4(1+|C|)}, & B^{2}<\min \left\{4(1+|C|)^{2},-4 A C\left(C^{-2}-1\right)\right\}, \\[2mm] R(A, B, C), & \text { otherwise, }\end{cases}
$$

where

$$
R(A, B, C)= \begin{cases}|A|+|B|+|C|, & |C|(|B|+4|A|) \leq|A B|, \\[2mm] -|A|+|B|+|C|, & |A B| \leq|C|(|B|-4|A|), \\[2mm] (|A|+|C|) \sqrt{1-\cfrac{B^{2}}{4 A C}}, & \text { otherwise. }\end{cases}
$$
\end{lemma}

\section{Main Results}\label{sec3}
\noindent We prove the following sharp inequality for second hankel determinant of inverse logarithmic coefficient for $\mathcal{K}_\alpha$.
\begin{theorem}
Let $f\in\mathcal{K}_{\alpha}$ given by \eqref{f} then
\begin{equation}\label{thm1}
|H_{2,1}(F_{f^{-1}}/2)|\leq
\begin{cases}
    \cfrac{\alpha^2}{36}, &  0< \alpha \leq 1/3,\\[5mm]
         \cfrac{\alpha^2(17+18\alpha+13\alpha^2)}{144(4+6\alpha+\alpha^2)},        & 1/3 < \alpha \leq 1.
\end{cases}
\end{equation}
	The inequality is sharp.

\end{theorem}

\begin{proof}
Fix $\alpha \in (0,1]$ and let $f\in \mathcal{K}_{\alpha}$ be of the form \eqref{f}. Then by \eqref{convexdef},
\begin{equation}\label{3.1.1}
1+\cfrac{zf''(z)}{f'(z)}=(p(z))^\alpha
\end{equation}
for some $p \in \mathcal{P}$ of the form \eqref{p}. By comparing the coefficients on both the sides of \eqref{3.1.1}, we obtain
	\begin{equation}\label{3.1.2}
		\begin{aligned}
			& a_2=\cfrac{\alpha}{2}\,\,\alpha c_1, \\[2mm]
			& a_3=\frac{\alpha}{12}(2c_2+(3\alpha-1)c_1^2), \\[2mm]
			& a_4=\frac{\alpha}{144}\left(12c_3+6(5\alpha-2)c_1c_2+(17\alpha^2-15\alpha+4)c_1^3 \right).
		\end{aligned}
	\end{equation}
Hence by \eqref{hankel}, we have
$$
H_{2,1}(F_{f^{-1}}/2)=\cfrac{\alpha^2}{2304}\left[24c_1c_3-16c_2^2-4(2+3\alpha) c_1^2 c_2+(4+6\alpha+\alpha^2)c_1^4  \right].
$$
Since the class $\mathcal{K}_\alpha$ is invariant under rotation[add one sentensece about lemma] and \eqref{invariance} holds, without loss of generality we can
assume that $c_1 = c$, where $0 \leq c \leq 2$.
\noindent Now using Lemma \ref{caratheodary} and straight forward computation
\begin{equation}\label{mainconvex}
\begin{aligned}
H_{2,1}(F_{f^{-1}}/2)
=&\frac{\alpha^2}{2304}\left[(2+\alpha^2)c^4-6\alpha(4-c^2)c^2 z-2(4-c^2)(8+c^2)z^2\right.\\[2mm]
&\left. +12c(4-c^2)(1-|z|^2)w\right].
\end{aligned}
\end{equation}

\noindent Now we consider different cases on $c$.
\\[2mm]
\textbf{Case 1.} Suppose that $c=0$. Then from \eqref{mainconvex}, for $\alpha \in (0,1]$,

\begin{equation}\label{convexc=0}
|\Gamma_1\Gamma_3-\Gamma_2^2|=\cfrac{\alpha^2}{36}\,\, |z|^2 \leq \cfrac{\alpha^2}{36}
\end{equation}

\vspace{2mm}
\noindent \textbf{Case 2.} Suppose that $c=2$. Then from \eqref{mainconvex}, for $\alpha \in (0,1]$,

\begin{equation}\label{convexc=2}
|\Gamma_1\Gamma_3-\Gamma_2^2| \leq \cfrac{\alpha^2(2+\alpha^2)}{144}
\end{equation}
\vspace{2mm}

\noindent \textbf{Case 3.} Suppose that $c\in (0,2)$. Since $|y|\leq 1$, from \eqref{mainconvex} we obtain

\begin{equation}\label{3cc}
\begin{aligned}
|\Gamma_1\Gamma_3-\Gamma_2^2|& \leq \frac{\alpha^2}{2304}\left[\left|(2+\alpha^2)c^4-6\alpha(4-c^2)c^2 z-2(4-c^2)(8+c^2)z^2\right|\right.\\[2mm]
&\left.\,\,\,\, +12c(4-c^2)(1-|z|^2)w\right]\\[2mm]
&=\cfrac{\alpha^2}{192}(c(4-c^2))\left[|A+Bz+Cz^2|+1-|z|^2\right]
\end{aligned}
\end{equation}
where

$$
A:=\cfrac{c^3(2+\alpha^2)}{12(4-c^2)},\,\,\,\,\,\,\, B:-\cfrac{c \alpha}{2},\,\,\,\,\,\,\, C:=-\cfrac{8+c^2}{6c}.
$$

\vspace{2mm}
\noindent Since $A C < 0$, we apply case (ii) ofLemma \ref{y(a,b,c)}. 

\vspace{2mm}
\noindent \textbf{3(a).} Note that the inequality
$$
-4 A C\left(\frac{1}{C^{2}}-1\right) \leq B^{2}
$$
is equivalent to
$$
\cfrac{c^4(7\alpha^2-4)+8c^2(8+13\alpha^2))}{8+c^2}\geq 0$$
which is evidently holds for $c \in (0,2).$ Moreover, the inequality $|B|<2(1-|C|)$ is equivalent to $c\,(16-2c+c^2(2+3\alpha))< 0$ which is false for $c \in(0,2)$.\\

\noindent \textbf{3(b).} Since 
$$
4(1+|C|)^2=\cfrac{c^4+52c^2+64}{9c^2} > 0,
$$
and
$$
-4 A C\left(C^{-2}-1\right)=-\cfrac{(2+\alpha^2)c^2(16-c^2)}{18(8+c^2)}<0.
$$
We see that 
$$
\min \left\{4(1+|C|)^{2},-4 A C\left(C^{-2}-1\right)\right\}=-4 A C\left(C^{-2}-1\right),
$$
and from case $3(a),$ we know that 
$$
-4 A C\left(C^{-2}-1\right) \leq B^{2}.
$$
Therefore, the inequality $B^{2} < \min \left\{4(1+|C|)^{2},-4 A C\left(C^{-2}-1\right)\right\}$ does not holds for $0<c<2$.\\

\noindent \textbf{3(c).} Next observe that the inequality, 
\begin{equation*}
|C|(|B|+4|A|) -|A B|=\cfrac{192\alpha^2+8(8-3\alpha+4\alpha^2)c^2+(8-12\alpha+4\alpha^2-3\alpha^3)c^4}{4-c^2} \leq 0
\end{equation*}
is equivalent to
\begin{equation}\label{3c} 
\varphi_1(c^2) \leq 0
\end{equation}
where
$$
\varphi_1(t)=192\alpha^2+8(8-3\alpha+4\alpha^2)x+(8-12\alpha+4\alpha^2-3\alpha^3)x^2,\,\,\,\,\,\, t\in (0,4).
$$

\noindent Note that $\varphi_1(0)>0$ and $\varphi_1(4)>0$. Also we can easily seen that $\varphi_1'(t)>0$ for $t \in (0,4)$. So $\varphi_1$ is increasing and hence $\varphi_1(t) >0$ in $t\in (0,4)$. Thus we deduce that the inequality \eqref{3c} is false.\\

\noindent \textbf{3(d).} Next note that the inequality
\begin{equation}\label{3d1}
|AB|-|C|(|B|-4|A|)=\cfrac{\alpha c^4(2+\alpha)^2}{24(4-c^2)}-\cfrac{8+c^2}{6c}\left(\cfrac{\alpha c}{2}-\cfrac{c^3(2+\alpha)^2}{3(4-c^2)}\right) \leq 0
\end{equation}
is equivalent to
\begin{equation}\label{3d2}
\varphi_2(u^2)\leq 0,
\end{equation}
where
$$
\varphi_2(u)=(8+12\alpha+4\alpha^2+3\alpha^3)u^2+8(8+3\alpha+4\alpha^2)u-192\alpha, \,\,\,\,\,\,\, u \in (0,4).
$$

\noindent We see that the discriminant $ \Delta := 64(64+144\alpha+217\alpha^2+72\alpha^3+52\alpha^4)>0.$
Thus we consider,

$$
u_{1,2}=\cfrac{4(-(4\alpha^2+3\alpha+8)\mp\sqrt{52\alpha^4+72\alpha^3+217\alpha^2+144\alpha+64}}{3\alpha^3+4\alpha^2+12\alpha+8}.
$$

\vspace{3mm}

\noindent It is clear that $u_1 <0$. Moreover $0<u_2<4$ holds. Indeed both inequalities $u_2>0$ and $u_2<4$ are equivalent to the evidently true inequalities 
$$
8+12\alpha+4\alpha^2+3\alpha^3 >0
$$
and
$$
192+33\alpha+264\alpha^2+264\alpha^3+102\alpha^4+48\alpha^5+9
\alpha^6 >0
$$
Thus \eqref{3d2} and so \eqref{3d1} is valid only when
$$
0<c\leq \tilde{c}=\sqrt{u_2}.
$$

\noindent Thus by \eqref{3cc} and Lemma \ref{y(a,b,c)},
\begin{equation}
\begin{aligned}
|\Gamma_1\Gamma_3-\Gamma_2^2| &\leq \cfrac{\alpha^2}{192}c(4-c^2)(-|A|+|B|+|C|)\\[2mm]
&=\cfrac{\alpha^2}{2304}(64-(8-24\alpha)c^2-(4+6\alpha+\alpha^2)c^4)\\[2mm]
&=\Phi(c^2)
\end{aligned}
\end{equation}
where
$$
\Phi(t)=\cfrac{\alpha^2}{2304}(64-(8-24\alpha)t-(4+6\alpha+\alpha^2)t^2), \,\,\,\,\,\,\,\, t \in (0,u_2).
$$
Since 
$$\Phi'(t)=\cfrac{\alpha^2}{2304}((8-24\alpha)-2(4+6\alpha+\alpha^2)c),\,\,\,\,\,\, t \in (0,u_2),
$$
we see that $0<\alpha \leq 1/3$, the funtion $\Phi$
is decreasing and so 
$$
\Phi(t)\leq\Phi(0)=\cfrac{\alpha^2}{36}, \,\,\,\,\, 0\leq t\leq u_2
$$
In the case $1/3 < \alpha \leq 1$,
$$t_0=\cfrac{4(3\alpha-1)}{4+6\alpha+\alpha^2}
$$
is a unique critical point of $\Phi$. Clearly $t_0>0$. It remains to check whether $t_0<u_2$, which is equivalent to 
\begin{equation*}
\begin{aligned}
\delta(\alpha)=& 117\alpha^8+240\alpha^7-149\alpha^6
-1212\alpha^5-4344\alpha^4 \\[2mm]
&-6288\alpha^3-4464\alpha^2
-1920\alpha-448 <0, \,\,\,\,\, \alpha \in (1/3,1],
\end{aligned} 
\end{equation*}
and since 
$$
\delta(\alpha)\leq -149\alpha^6
-1212\alpha^5-4344\alpha^4 
-6288\alpha^3-4464\alpha^2
-1920\alpha-91 <0
$$
for $\alpha \in (1/3,1]$, we deduce that $t_0 <u_2$.\\

\noindent Thus for $1/3 < \alpha \leq 1$, we have
$$
\Phi(t) \leq \Phi(t_0)=\cfrac{\alpha^2(17+18\alpha+13\alpha^2)}{144(4+6\alpha+\alpha^2)}, \,\,\,\,\,\, 0<t<u_2.
$$
We can conclude that, for $0<c\leq\tilde{c}$
\begin{equation}\label{3cconclusion}
|\Gamma_1\Gamma_3-\Gamma_2^2| \leq
\begin{cases}
    \cfrac{\alpha^2}{36}, &  0< \alpha \leq 1/3,\\[4mm]
         \cfrac{\alpha^2(17+18\alpha+13\alpha^2)}{144(4+6\alpha+\alpha^2)},        & 1/3 < \alpha \leq 1.
\end{cases}
\end{equation}

\noindent \textbf{3(d).} We now consider the last case in Lemma \ref{y(a,b,c)}, which in view of 3(d) holds for $\tilde{c}<c<2$. Then by \eqref{3cc}, we have 
\begin{equation}\label{lastconvex}
\begin{aligned}
|\Gamma_1\Gamma_3-\Gamma_2^2| &\leq \cfrac{\alpha^2}{192}\,\,c\,\,(4-c^2)(|C|+|A|)\sqrt{1-\cfrac{B^2}{4AC}}\\[2mm]
&=\cfrac{\alpha^2}{2304\sqrt{2(2+\alpha^2)}}\left(64-8c^2+\alpha^2c^4\right)\sqrt{\cfrac{32+52\alpha^2+c^2(4-7\alpha^2)}{8+c^2}}\\[2mm]
&=\cfrac{\alpha^2}{2304\sqrt{2(2+\alpha^2)}}g_1(c^2)\sqrt{g_2(c^2)}
\end{aligned}
\end{equation}
where
$$
g_1(x)=64-8x+\alpha^2x^2
$$
and
$$
g_2(x)=\cfrac{32+52\alpha^2+c^2(4-7\alpha^2)}{8+c^2}
$$
It is easily seen that $g_1$ and $g_2$ are decreasing on $(u
_2,4)$, and so from \eqref{lastconvex} we obtain
\begin{equation}\label{3dd}
|\Gamma_2\Gamma_3-\Gamma_1^2|  \leq\cfrac{\alpha^2}{2304\sqrt{2(2+\alpha^2)}}g_1(u_2)\sqrt{g_2(u_2)}=\Psi(u_2)
\end{equation}
%

\noindent \textbf{Case 4.} It remains to compare the bounds in \eqref{convexc=0},\eqref{convexc=2},\eqref{3cconclusion} and \eqref{3dd}. The inequality

$$
\cfrac{\alpha^2(2+\alpha^2)}{144} \leq \cfrac{\alpha^2}{36}, \,\,\,\,\, \alpha \in(0,1] 
$$

\noindent is trivial. Since the function is $\Phi$ is decreasing in $(u_2,4)$ and $\Phi(u_2)=\Psi(u_2)$, the inequality 
$$
\cfrac{\alpha^2(2+\alpha^2)}{144}\leq \cfrac{\alpha^2(17+18\alpha+13\alpha^2)}{144(4+6\alpha+\alpha^2)}
$$
is true for $1/3 < \alpha \leq 1$. Finally the inequality
$$
\cfrac{\alpha^2}{36}\leq \cfrac{\alpha^2(17+18\alpha+13\alpha^2)}{144(4+6\alpha+\alpha^2)}, \,\,\,\,\, \alpha \in (1/3,1]
$$
is equivalent to 
$$
9\alpha^2-6\alpha+1\geq0 
$$
which is evidently true for $1/3 < \alpha \leq 1$.\\
Thus summarizing the results in cases 1-4, we see that \eqref{thm1} holds. \\

We now proceed to prove the equality part. When $\alpha \in (0,1/3]$, equality holds for the function $f \in \mathcal{A}$ given by \eqref{3.1.1} with $p
$ given by 
$$
p(z)=\cfrac{1+z}{1-z}, \,\,\,\,\,\,\,\, z \in \mathbb{D}.
$$
When $\alpha \in (1/3,1]$, set
$$
\tau:=\sqrt{t_0}=\sqrt{\cfrac{4(3\alpha-1)}{(4+6\alpha+\alpha^2)}}.
$$

\noindent Since $\tau \leq 2$, the function
$$
\tilde{p}(z)=\cfrac{1-\tau z+z^2}{1-z^2}, \,\,\,\,\, z \in \mathbb{D},
$$

\noindent belongs to $\mathcal{P}$. Thus the function $f$ given by \eqref{3.1.1}, where $p$ replaced by $\tilde{p}$ belongs to $\mathcal{K}_\alpha$ and is extremal function for the second inequality in \eqref{thm2}

where . This completes the proof.
\end{proof}
For $\alpha=1$, we get the following result\cite{amal}.
\begin{corollary}
If $f \in \mathcal{K}$, then 
$$
|H_{2,1}(F_{f}/2)|\leq \cfrac{1}{33}.
$$
\end{corollary}

We next find the sharp bound for the second Hankel determinant of inverse logarithmic coefficient of functions in $ \mathcal{S}^*_\alpha$.

\begin{theorem}
Let $f\in\mathcal{S}^*_{\alpha}$ given by \eqref{f} then
\begin{equation}\label{thm2}
|H_{2,1}(F_{f^{-1}}/2)|\leq
\begin{cases}
    \cfrac{\alpha^2}{4}, &  0< \alpha < 1/5,\\[4mm]
         \cfrac{\alpha^2(2+5\alpha+15\alpha^2)}{7+30\alpha+35\alpha^2},        & 1/5 \leq \alpha \leq \alpha', \\[4mm]
         \cfrac{\alpha^2}{36}(4+35\alpha^2), & \alpha'<\alpha \leq 1.
\end{cases}
\end{equation}
where $\alpha'=0.390595...$ is the unique root in $(0,1)$ of the equation $
44+60\alpha+155\alpha^2-1050\alpha^3-1225\alpha^4
=0$. The inequality \eqref{thm2} is sharp.

\end{theorem}

\begin{proof}
Fix $\alpha \in (0,1]$ and let $f\in \mathcal{S}^*_{\alpha}$ be of the form \eqref{f}. Then by \eqref{stardef},
\begin{equation}\label{3.3.1}
\cfrac{zf'(z)}{f(z)}=(p(z))^\alpha
\end{equation}
for some $p \in \mathcal{P}$ of the form \eqref{p}. By comparing the coefficients on both sides of \eqref{3.3.1}, we obtain
	\begin{equation}\label{3.3.2}
		\begin{aligned}
			& a_2=\alpha c_1, \\[2mm]
			& a_3=\frac{\alpha}{4}(2c_2+(3\alpha-1)c_1^2), \\[2mm]
			& a_4=\frac{\alpha}{36}\left(12c_3+6(5\alpha-2)c_1c_2+(17\alpha^2-15\alpha+4)c_1^3 \right).
		\end{aligned}
	\end{equation}
Hence by \eqref{hankel}, we have
$$
H_{2,1}(F_{f^{-1}}/2)=\cfrac{\alpha^2}{576}\left((7+30\alpha+35\alpha^2) c1^4-(1+5\alpha) c_1^2c_2-36c_2^2+48 c_1c_3\right).
$$

\noindent Since the class $\mathcal{S}^*_{\alpha}$ is invariant under rotation, without loss of generality, we can
assume that $c_1 = c,$ where $0 \leq c \leq 2$.  Therefore, by Lemma \ref{caratheodary}, for some $c \in [0,2]$ and $z,w \in \mathbb{D}$ we have

\begin{equation}\label{main}
\begin{aligned}
H_{2,1}(F_{f^{-1}}/2)
=&\frac{\alpha^2}{576}\left[(4+35\alpha^2)c^4-30\alpha(4-c^2)c^2 z-3(4-c^2)(12+c^2)z^2\right.\\[2mm]
&\left. +24c(4-c^2)(1-|z|^2)w\right].
\end{aligned}
\end{equation}
Now we have the following cases on $c$.\\[1mm]
\textbf{Case 1:} Suppose that $c=0$. Then by \eqref{main}, we obtain
\begin{equation}\label{starlikec=0}
|H_{2,1}(F_{f^{-1}}/2)|=\frac{1}{4}|z^2|\alpha^2\leq\frac{\alpha^2}{4}.
\end{equation}
\textbf{Case 2:} Suppose that $c=2$. Then by \eqref{main}, we obtain
\begin{equation}\label{convexc=2}
|H_{2,1}(F_{f^{-1}}/2)|=\frac{\alpha^2}{36}(4+35\alpha^2).
\end{equation}
\textbf{Case 3:} Suppose that $c \in (0,2)$. Applying the triangle inequality in \eqref{main} and by using the fact
that $|w| \leq 1$, we obtain
\begin{equation}\label{case3main}
\begin{aligned}
H_{2,1}(F_{f^{-1}}/2)
&=\frac{\alpha^2}{576}\left[\left|(4+35\alpha^2)c^4-30\alpha(4-c^2)c^2 z-3(4-c^2)(12+c^2)z^2\right|\right.\\[2mm]
&\left. \,\,\,\, + 24c(4-c^2)(1-|z|^2)w\right] \\[2mm]
&\leq \frac{\alpha^2}{24}c(4-c^2)\left[\left| A+B z+C z^2 \right|+1-|z^2|\right],
\end{aligned}
\end{equation}
where 
$$
A:=\frac{c^3 (4 + 35 \alpha^2)}{24(4-c^2)}, \quad B:=-\frac{5}{4} \alpha c,\quad C:=-\frac{12+c^2}{8c}.
$$
Since $AC < 0$, so we can apply case (ii) of Lemma \ref{y(a,b,c)}.\\
[3mm]
\textbf{3(a).} Note that the inequality
$$
-4 A C\left(\frac{1}{C^{2}}-1\right) \leq B^{2}
$$
is equivalent to
$$
\cfrac{c^2(36+540\alpha^2+c^2(10\alpha^2-1))}{12+c^2}\geq 0$$

\vspace{2mm}
\noindent which evidently holds for $c \in (0,2).$ Moreover, the inequality $|B|<2(1-|C|)$ is equivalent to $12 + c^2 (1 + 5 \alpha)-8c<0$ which is not true for $c \in(0,2)$.\\
[3mm]
\textbf{3(b).} 
Since 
$$
-4 A C\left(C^{-2}-1\right)=\cfrac{c^2(-36+c^2)(4+35\alpha^2)}{48(12+c^2)} < 0,
$$
and
$$
4(1+|C|)^2=\frac{(12-8c+c^2)^2}{16c^2}>0.
$$
So we get 

$$
\min \left\{4(1+|C|)^{2},-4 A C\left(\frac{1}{C^{2}}-1\right)\right\}=-4 A C\left(\frac{1}{C^{2}}-1\right),
$$
and from $3(a),$ we know that 
$$
-4 A C\left(\frac{1}{C^{2}}-1\right) \leq B^{2}.
$$
Therefore, the inequality $B^{2} < \min \left\{4(1+|C|)^{2},-4 A C\left(\frac{1}{C^{2}}-1\right)\right\}$ does not holds for $0<c<2$.\\
[3mm]
\textbf{3(c).} Note that the inequality $$|C|(|B|+4|A|) \leq |A B|$$ is equivalent to 

$$
720\alpha+24c^2(4-5\alpha+35\alpha^2)-c^4(-8+35\alpha-70\alpha^2+175\alpha^3) \leq 0.
$$
Consider the function $\phi_1:(0,4)\rightarrow \mathbb{R}$ defined by 
\noindent
$$
\phi_1(x)=720\alpha+24x(4-5\alpha+35\alpha^2)-x^2(-8+35\alpha-70\alpha^2+175\alpha^3)
$$

\noindent Clearly $\phi_1(0)=720 \alpha > 0$ and $\phi_1(4)=16(32-20\alpha+280\alpha^2-175\alpha^3>0$. It can be shown that $\phi_1'(x)>0$. Hence $\phi_1$ is increasing and consequently we concluded that $\phi_1(x)>0$. Therefore the inequality $|C|(|B|+4|A|) \leq |A B|$ is false.
\vspace{2mm}

\noindent \textbf{3(d).}
Note that the inequality 

\begin{equation}\label{eq1}
\begin{aligned}
&|A B|-|C|(|B|-4|A|)\\[2mm]
&= \cfrac{5c^4\alpha (4+35\alpha^2)}{96(4-c^2)}-\cfrac{12+c^2}{8c}\left(\cfrac{5 c \alpha}{4}-\cfrac{c^3(4+35\alpha^2}{6(4-c^2}\right)
\end{aligned}
\end{equation}
is equivalent to 
\begin{equation}\label{eq2}
\delta(c^2)\geq 0,
\end{equation}
where
\begin{equation*}
\delta(t)=720 \alpha-24(4+5\alpha+35\alpha^2)t-(8+35\alpha+70\alpha^2+175\alpha^3)t^2, \,\,\, t \in (0,4)
\end{equation*}

\noindent We see that for $\alpha \in (0,1]$
$$
4+5\alpha+35\alpha^2 >0, \,\, 8+35 \alpha+70\alpha^2+175\alpha^3 >0,
$$
and the discriminant $\Delta:=2304(4+20 \alpha+120\alpha^2+175\alpha^3+525\alpha^4]>0$ for $\alpha \in (0,1].$\\

\noindent Thus we consider
$$
t_{1,2}=-\cfrac{12\left(4+5\alpha+35\alpha^2\pm 2\sqrt{4+20\alpha+120\alpha^2+175\alpha^3+525\alpha^4}\right)}{8+35\alpha+70\alpha^2+175\alpha^3}
$$
\vspace{1mm}

\noindent It is clear that $t_1 <0$ and so it remains to check if $0<t_2<4$. The inequality $t_2>0$ is equivalent to 

$$
4+5\alpha+35\alpha^2-2\sqrt{4+20\alpha+120\alpha^2+175\alpha^3+525\alpha^4}<0
$$
 
\noindent which is true for $\alpha \in (0,1]$. Further the inequality $t_2 <4 $ can be written as

$$
6\sqrt{4+20\alpha+120\alpha^2+175\alpha^3+525\alpha^4}<5(4+10\alpha+35\alpha^2+35\alpha^3),
$$

\noindent which is also true for $\alpha \in (0,1]$. There for \eqref{eq2} and so \eqref{eq1} is valid for 
$$
0<c\leq c^*:=\sqrt{t_2}.
$$

\noindent Then by Lemma 2 and \eqref{eq1}, for $0<c\leq c^*$, we have 

\begin{equation*}
\begin{aligned}
|\Gamma_1\Gamma_3-\Gamma_2^2| & \leq  \cfrac{\alpha^2}{24}c(4-c^2)\left(-|A|+|B|+|C|\right)\\[2mm]
&=\cfrac{\alpha^2}{576}\left(144-24(1-5\alpha)c^2-(70+30\alpha+35\alpha^2)c^4\right)\\[2mm]
&=\cfrac{\alpha^2}{576}\phi_2(c^2)
\end{aligned}
\end{equation*}
where 

$$\phi_2(s)=144-24(1-5\alpha)s-(70+30\alpha+35\alpha^2)s^2, \,\,\,\, 0 < s \leq t_1.
$$
\\[1mm]
we note that $\phi'(s)=0$ only when $s=s_0:=12(5\alpha-1)/(7+30\alpha+35\alpha^2)$ and it is easy to see that $s_0>0$ only if $\alpha>1/5$. Further $s_0 \leq t_2$ is true if $\beta(\alpha)<0$, \\
where
\begin{equation*}
\begin{aligned}
\beta(x)=&-96-1060x-7040x^2-29975x^3-77525x^4\\
&-124250x^5-82250x^6+153125x^7+459375x^8, \,\,\,\, 0 < x \leq 1
\end{aligned}
\end{equation*}

\noindent So we get $0<s_0<t_2$ if and only if $1/5<\alpha\leq \alpha_0$, where $\alpha_0$ is the unique positive real root of the polynomial $\beta(x)$ lies between $0$ and $1$.
Since the leading coefficient of $\phi(x)$ is negative, for $1/5 < \alpha < \alpha_0$, we get
\begin{equation}\label{alpha1}
\begin{aligned}
|\Gamma_1\Gamma_3-\Gamma_2^2|
&\leq \cfrac{\alpha^2}{576}\phi_2(s_0)\\[2mm]
&=\alpha^2\left(\cfrac{2+5\alpha+15\alpha^2}{7+30\alpha+35\alpha^2}\right).
\end{aligned}
\end{equation}
For $0<\alpha \leq 1/5$, it is clear that $\phi_2$ is decreasing, so we get 
\begin{equation}\label{alpha2}
\begin{aligned}
|\Gamma_1\Gamma_3-\Gamma_2^2|
&\leq \cfrac{\alpha^2}{576}\,\,\phi_2(0)\\[2mm]
&=\cfrac{\alpha^2}{4}.
\end{aligned}
\end{equation}
For $\alpha_0<\alpha \leq 1$, it is clear that $\phi_2$ is increasing, so we get 
\begin{equation}\label{alpha3}
\begin{aligned}
|\Gamma_1\Gamma_3-\Gamma_2^2|
&\leq \cfrac{\alpha^2}{576} \,\, \phi_2(t_1)
\end{aligned}
\end{equation}
\textbf{3(e).} Next suppose that $c \in (c^*,2)$. Then by Lemma \ref{y(a,b,c)}, we have 
\begin{equation}
\begin{aligned}
|\Gamma_1\Gamma_3-\Gamma_2^2|
&\leq \frac{\alpha^2}{576}c(4-c^2)(|C|+|A|) \sqrt{1-\frac{B^{2}}{4 A C}}\\
&= \cfrac{\alpha^2}{288\sqrt{4+35\alpha^2}}\left(144-24c^2+(1+35\alpha^2)c^4\right)\sqrt{\cfrac{12+180\alpha^2+(1-10\alpha^2)c^2}{12+c^2}}\\
&=\cfrac{\alpha^2}{288\sqrt{4+35\alpha^2}} \,\,\,h_1(c^2)\sqrt{h_2(c^2)}
\end{aligned}
\end{equation}
where $h_1$ and $h_2$ are defined by 
$$
h_1(x)=144-24x+(1+35\alpha^2)x^2
$$
and
$$
h_2(x)=\cfrac{12+180\alpha^2+(1-10\alpha^2)x}{12+x}.
$$

\noindent Now we consider the function  $F(x)=h_1(x)\sqrt{h_2(x)}$ on $(t_1,4)$ and we show that $F$ is convex function. It is enough to show that $F''(x)\geq 0.$\\
Since 
\begin{equation*}
\begin{aligned}
F''(x)(h_2(x))^{3/2}&=h_1''(x)(h_2(x))^2+h_1'(x)h_2(x)h_2'(x)-\frac{1}{4}h_1(x)(h_2'(x))^2+\frac{1}{2}h_1(x)h_2(x)h_2''(x)\\
&=\cfrac{1}{(12+x)^4}\left[2592(8+820\alpha^2+14075\alpha^4+63000\alpha^6)
\right.\\  &\left.\,\,\,\,-432(-16-890\alpha^2-4525\alpha^4+31500\alpha^6)x
\right.\\
&\left.\,\,\,\,-
54(-16-540\alpha^2+1475\alpha^4+31500\alpha^6)x^2\right.\\
&\left.-\,\,\,\,
6(8+195\alpha^2-2925\alpha^4+1750\alpha^6)x^3+
(1-10\alpha^2)^2(1+35\alpha^2)x^4\right]\\
&=\cfrac{1}{(12+x)^2}G(x)
\end{aligned}
\end{equation*}
It is easy to see that $h_2$ is a positive decreasing function in $(t_2,4)$. We show that our assertion is true by proving that $G(x)\geq 0$ for $x \in (t_2, 4).$\\[2mm]
Let $x \in (t_2,4)$ is fixed and 
\begin{equation}
H(\alpha)=m_0+m_1\alpha^2+m_2\alpha^4+m_3\alpha^6
\end{equation}
where
\begin{equation}
\begin{aligned}
&m_0:=20736+6912x+864x^2+48x^3+x^4, \\[2mm]
&m_1:=2125440+384480x+29160x^2+1170x^3+15x^4, \\[2mm]
&m_2:=36482400+1954800x-79650x^2-17550x^3-600x^4, \\[2mm]
&m_3:=163296000-13608000x-146750x^2+10500x^3+3500x^4.\\
\end{aligned}
\end{equation}
Clearly $m_0>0 $ and $m_1>0$. For $x \in (x_1,4)$, we have
\begin{equation}
\begin{aligned}
m_2&=36482400+1954800x-79650x^2-17550x^3-600x^4, \\[2mm]
&> 1200(28276+1629x)>0
\end{aligned}
\end{equation}
Similarly, we have
\begin{equation}
\begin{aligned}
m_3&=163296000-13608000x-146750x^2+10500x^3+3500x^4.\\[2mm]
&>3500 (24408+3x^3+x^4)>0
\end{aligned}
\end{equation}
So we get $H(\alpha) >0$ for $\alpha \in (0,1]$ and $x \in (t_2,4)$, which therefore shows that $G(x) \geq 0.$
So we can conclude that

\begin{equation*}
\begin{aligned}
F(x)&\leq \,\, \max \{F(t_2),F(4)\} \\[2mm]
&\leq 
\begin{cases}
    F(t_2),&  0< \alpha \leq \alpha^*\\[2mm]
    F(4),              & \alpha^* \leq \alpha \leq 1
\end{cases}
\end{aligned}
\end{equation*}
So we get, for $c \in (c^*,2)$
\begin{equation}\label{alpha4}
|\Gamma_1\Gamma_3-\Gamma_2^2| \leq
\begin{cases}
    \cfrac{\alpha^2}{288\sqrt{4+35\alpha^2}}\,\,F((c^*)^2), &  0< \alpha \leq \alpha^*\\[4mm]
    \cfrac{\alpha^2}{36}(4+35\alpha^2),              & \alpha^* \leq \alpha \leq 1
\end{cases}
\end{equation}
\textbf{Case 4.} Now we compare the bounds in \eqref{alpha1},\eqref{alpha2},\eqref{alpha3} and \eqref{alpha4}.
The inequality
$$
\cfrac{\alpha^2}{36}(4+35\alpha^2) \leq \cfrac{\alpha^2}{4}
$$
is true for $0<\alpha\leq 1/\sqrt{7}$. So we can conclude that for $0<\alpha<1/5$,
$$
|H_{2,1}(F_{f^{-1}}/2)|\leq \cfrac{\alpha^2}{4}.
$$
The inequality 
$$
\cfrac{\alpha^2}{4}\leq \cfrac{\alpha^2(2+5\alpha+15\alpha^2)}{7+30\alpha+35\alpha^2}
$$
is equivalent to the evidently true inequality $(1-5\alpha)^2\geq 0.$ A tedious long calculation shows that the following inequality 

$$
\cfrac{\alpha^2}{288\sqrt{4+35\alpha^2}}\,\,F((c^*)^2) \leq \cfrac{\alpha^2(2+5\alpha+15\alpha^2)}{7+30\alpha+35\alpha^2},
$$

\vspace{3mm}
\noindent is true for $0<\alpha\leq 1$. The inequality

$$
\cfrac{\alpha^2}{36}(4+35\alpha^2) \leq \cfrac{\alpha^2(2+5\alpha+15\alpha^2)}{7+30\alpha+35\alpha^2}
$$
is equivalent to 
$$
44+60\alpha+155\alpha^2-1050\alpha^3-1225\alpha^4 \geq 0,
$$

\vspace{2mm}
\noindent which is true for $0\leq \alpha \leq \alpha'$, where $\alpha'$ is the unique positive real root of the polynomial  $44+60\alpha+155\alpha^2-1050\alpha^3-1225\alpha^4
$ lies between $0$ and $1$. With further long  computations, we can show that the inequality 

$$
\cfrac{\alpha^2}{576} \,\, \phi_2(t_1) \leq \cfrac{\alpha^2}{36}(4+35\alpha^2)
$$
is true for $2/3 \leq \alpha \leq 1$. By summarizing the above cases we get \eqref{thm2} holds. \\

In order to show that the inequalities in \eqref{thm2} are sharp. For $0< \alpha <1/5$, consider the function
$$p(z)=\cfrac{1+z^2}{1-z^2}, \,\,\,\,\,\, z \in \mathbb{D}.$$
It is obvious that $p$ is in $\mathcal{P}$ with $c_1=c_3=0$ and $c_2=2$. The corresponding function $ f \in \mathcal{S}^*_\alpha$ is described by \eqref{3.3.1}. Hence by \eqref{3.3.2}, it follows that $a_2=a_4=0$ and $a_3=\alpha$. From \eqref{hankel}, we obtain 
$$
|H_{2,1}(F_{f^{-1}}/2)|=\cfrac{\alpha^2}{4}.
$$
For $1/5 \leq \alpha \leq \alpha'$, consider the function $f \in \mathcal{A}$ given by \eqref{3.3.1} with $p$ given by
$$
p(z)=\cfrac{1-\zeta z+z^2}{1-z^2}, \,\,\,\,\,\,\, z \in \mathbb{D},
$$
where $\zeta:=\sqrt{s_0}=\sqrt{12(5\alpha-1)/(7+30\alpha+35\alpha^2)}.$ By simple computation we can show that 
$$
|H_{2,1}(F_{f^{-1}}/2)|= \cfrac{\alpha^2(2+5\alpha+15\alpha^2)}{7+30\alpha+35\alpha^2}.
$$
For the last case, $\alpha' <\alpha \leq 1$ equality holds for the function $f \in \mathcal{A}$ of the form \eqref{3.3.1} with $p(z)=1+z/1-z$. This completes the proof.
\end{proof}
For $\alpha=1$, we get the estimate for the class $\mathcal{S}^*$ of starlike functions \cite{vibhuthi2}.
\begin{corollary}
If $f \in \mathcal{S}^*$, then
$$
|H_{2,1}(F_{f^{-1}}/2)|\leq \cfrac{13}{12}.
$$
\end{corollary}

{\bf Acknowledgment:}
The research of the first named author is supported by SERB-CRG, Govt. of India and the second named author's research work is supported by UGC-SRF.

\end{document}